\documentclass[a4paper,11pt]{amsart}
\usepackage{amssymb,amsthm,bbm,epic,eepic,graphics,amsbsy}
\usepackage{amssymb,amsthm}
\usepackage{xypic}
\usepackage[applemac]{inputenc}
\newcommand{\la}{\lambda}

\def\cl{\mathcal{C}}

\def\Hom{\mathrm{Hom}}

\def\End{\mathrm{End}}
\def\Ker{\mathrm{Ker}\,}
\def\Res{\mathrm{Res}}

\def\Aut{\mathrm{Aut}}

\def\C{\ensuremath{\mathbbm{C}}}

\def\Z{\mathbbm{Z}}

\def\R{\mathbbm{R}}

\def\gl{\mathfrak{gl}}
\def\sl{\mathfrak{sl}}

\def\osp{\mathfrak{osp}}

\def\k{\mathbbm{k}}
\def\kt{\mathbbm{k}^{\times}}

\def\om{\omega}

\def\eps{\epsilon}

\def\into{\hookrightarrow}
\def\ii{\mathrm{i}}

\def\GL{\mathrm{GL}}

\newtheorem{theo}{Theorem}[section]

\newtheorem{prop}[theo]{Proposition}
\newtheorem{defi}[theo]{Definition}
\newtheorem{lemma}[theo]{Lemma}

\newtheorem{cor}[theo]{Corollary}

\def\mod{\ \mathrm{mod}\ }

\makeatletter
\DeclareRobustCommand\widecheck[1]{{\mathpalette\@widecheck{#1}}}
\def\@widecheck#1#2{%
   \setbox\z@\hbox{\m@th$#1#2$}%
   \setbox\tw@\hbox{\m@th$#1%
      \widehat{%
         \vrule\@width\z@\@height\ht\z@
         \vrule\@height\z@\@width\wd\z@}$}%
   \dp\tw@-\ht\z@
   \@tempdima\ht\z@ \advance\@tempdima2\ht\tw@ \divide\@tempdima\thr@@
   \setbox\tw@\hbox{%
      \raise\@tempdima\hbox{\scalebox{1}[-1]{\lower\@tempdima\box\tw@}}}%
   {\ooalign{\box\tw@ \cr \box\z@}}}
\makeatother

\title{Group algebras of finite groups as Lie algebras}
\author{Ivan Marin}
\date{January 11th, 2008}
\begin{document}

\maketitle


\begin{center}
 Institut de Math\'ematiques de Jussieu\\ Université
Paris 7 \\ 175 rue du Chevaleret\\ F-75013 Paris \\
\end{center}

\bigskip
\bigskip

\noindent {\bf Abstract.} We consider the natural Lie algebra
structure on the (associative) group algebra of a finite
group $G$, and show that the Lie subalgebras associated to
natural involutive antiautomorphisms of this group
algebra are reductive ones. We give a decomposition
in simple factors of these Lie algebras, in terms of the ordinary representations of
$G$.
\medskip


\section{Introduction}

Let $G$ be a \emph{finite} group and $\widehat{G}$ its set
of ordinary irreducible representations up to isomorphism.
Let $\k$ be a field of characteristic 0 such that each
ordinary representation of $G$ is defined over $\k$
(for instance any $\k$ containing the field
of cyclotomic numbers).

As a Lie algebra, the group algebra $\k G$ is reductive,
and the canonical decomposition of $\k G$ as a semisimple
unital algebra translates into a decomposition as a Lie algebra
$$
\k G = \bigoplus_{V \in \widehat{G}} \gl(V).
$$
In particular, $\k G$ is a reductive Lie algebra.
Likewise, the center of $\k G$ as a Lie algebra is the same
as its center as a group algebra, and is generated by the elements
$T_c$, where $T_c$ for $c$ a conjugacy class of $G$ is defined
as the sum of all elements of $c$. For any $g$ in $G$ we let $c(g)$ denote
its conjugacy class. There is a projection
$p : \k G \to Z(\k G)$ defined by
$$
p(x) = \frac{1}{\# G} \sum_{g \in G} gxg^{-1}
$$
whose kernel can also be described as the intersection of the kernels
of the linear forms $\delta_c$, defined for $c$ a conjugacy class
by $\delta_c(g) = 1$ if $g \in c$ and $\delta_c(g) = 0$ if $g \in G \setminus
c$. Since all $[g_1,g_2] = g_1 g_2 - g_1^{-1}(g_1g_2)g_1$,
for $g_1,g_2 \in G$, belong to these kernels, it follows that
$$
(\k G) ' = \bigcap \Ker \delta_c = \Ker p = \bigoplus_{V \in \widehat{G}}
\sl(V).
$$

\subsection{Definition of the Lie algebras $\mathcal{L}_{\alpha}(G)$}

The purpose of this note is to show that one can go
at least one step further in the understanding of the Lie-theoretical
aspects of this structure. Indeed, a classical way to get Lie subalgebras
of an associative algebra $A$, is to considerer involutive antiautomorphisms
$S$ of $A$. Then, $\{ a \in A \ | \ S(a) = -a \}$ is a Lie
subalgebra of $A$, which is spanned by the elements $x-S(x)$
for $x$ belonging to some basis of $A$.

In the case of group algebras, there is a canonical antiautomorphism
$g \mapsto g^{-1}$. More generally, if $\alpha : G \to \k^{\times}$
is a multiplicative character, then $S : g \mapsto \alpha(g) g^{-1}$
extends to an involutive antiautomorphism of $G$. We denote $\mathcal{L}_{\alpha}(G)$
the corresponding Lie algebra. According to the above remarks, it
can be defined as follows.

\begin{defi}
For $\alpha : G \to \k^{\times}$ a multiplicative character
of $G$, we denote $\mathcal{L}_{\alpha}(G)$ the
Lie subalgebra of $\k G$ spanned by the elements $g - \alpha(g) g^{-1}$
for $g \in G$.
\end{defi}

If $\alpha = 1$ is the trivial character, then $S$ is the antipode of the
natural Hopf algebra structure of $\k G$. We let $\mathcal{L}(G)
= \mathcal{L}_1(G)$. The correspondence $G \rightsquigarrow \mathcal{L}(G)$
defines a left exact functor from (finite) groups to Lie algebras.
Note that $\mathcal{L}(G) = \{ 0 \}$ iff $G$
is isomorphic to some $(\Z/2\Z)^r$. We will see that the structure
of $\mathcal{L}(G)$ is closely related to the structure of the group algebra $\R G$.

In general, $S$ is an antiautomorphism of $\k G$ as a symmetric
algebra, namely $t \circ S = t$, where $t$ is the usual trace
$t(g) = 1$ if $g = e$, $t(g) = 0$ if $g\neq e$, with $e$ the neutral
element of $G$. In particular
$\mathcal{L}_{\alpha}(G)$ is orthogonal to the subspace of
invariants $(\k G)^{S}$ with respect to the bilinear form $(a,b) \mapsto
t(ab)$.

Another way to see these Lie algebras is the following one.
For $g_1,g_2 \in G$, the formula
$(g_1,g_2) = \alpha(g_1) \delta_{g_1,g_2} = \alpha(g_2) \delta_{g_1,g_2}
$ defines a nondegenerate bilinear form on $\k G$. Embedding
$\k G$ in $\End(\k G)$ by left multiplication, one gets
that the adjoint of $g \in G$ with respect to $(\ , \ )$
is $\alpha(g) g^{-1}$. Hence $\mathcal{L}_{\alpha}(G)$ can be
viewed inside $\End(\k G)$ as the intersection of the
corresponding orthogonal Lie algebra and of the image of $\k G$.

\subsection{Main result}

Letting $p_{\alpha}$ linearly extending the natural map $g \mapsto
\frac{g - \alpha(g) g^{-1}}{2}$. It is readily checked that $p_{\alpha}$
is a projector on $\mathcal{L}_{\alpha}(G)$. Moreover one easily
gets, for instance by computing the trace of $p_{\alpha}$, that
$$
\dim \mathcal{L}_{\alpha}(G) = \frac{1}{2} \# \{ g \in G \ | \ g^2 \neq 1 \} 
+ \# \{ g \in G \ | \ g^2 = 1, \alpha(g) \neq 1 \}
$$

To such a character $\alpha$ one can associate another
Lie algebra. Let $V \in \widehat{G}$. If $\alpha \hookrightarrow
V \otimes V$ as a representation of $G$, by semisimplicity
this gives rise to a bilinear form $V \otimes V \to \k$, which is
nondegenerate by irreducibility of $V$. Let $\osp(V)$
be the Lie subalgebra of $\gl(V)$ leaving this form invariant.
It is easily checked that the component of $g - \alpha(g) g^{-1}$
on $\gl(V)$ for $g \in G$ belongs to $\osp(V)$.

On the contrary, if $\alpha$ does not inject in $V \otimes V$, this
is equivalent to saying that $V^* \otimes \alpha$ is not
isomorphic to $V$. It follows that $\gl(V) \oplus \gl(V^* \otimes \alpha)$
lies inside $\k G$. Let $\gl_{\alpha}(V)$
be the image of $\gl(V)$ under the map $x \mapsto (x,- ^t x)$.
Then the component of $g - \alpha(g) g^{-1}$
on $\gl(V)\oplus \gl(V^* \otimes \alpha)$ for $g \in G$ belongs to $\gl_{\alpha}(V)$.
Note that $\gl_{\alpha}(V) = \gl_{\alpha}(V^{*} \otimes \alpha)$.

It is then convenient to introduce the equivalence relation
on $\widehat{G}$ generated by $V \sim V^* \otimes \alpha$.
We denote $\widehat{G}_{\alpha} = \widehat{G}/\sim$ its set
of equivalence classes. Let $\widehat{G}_{\alpha}^{even}$ and
$\widehat{G}_{\alpha}^{odd}$ be the set of elements of $\widehat{G}_{\alpha}$
of cardinal 1 and 2, respectively. To any class $\{ V \}$ or
$\{ V, V^* \otimes \alpha \}$ we associated a well-defined
Lie subalgebra of $\k G$, $\osp_{\alpha}(V)$ or $\gl_{\alpha}(V)$.
Gluing all these together, we get the following Lie algebra.
\begin{defi}
For $\alpha : G \to \k^{\times}$ a multiplicative character
of $G$, then $\mathcal{M}_{\alpha}(G)$ is defined as the
Lie subalgebra
$$
\mathcal{M}_{\alpha}(G) = \left( \bigoplus_{V \in \widehat{G}_{\alpha}^{even}} \osp_{\alpha}(V) \right) \oplus
\left( \bigoplus_{V \in \widehat{G}_{\alpha}^{odd}} \gl_{\alpha}(V) \right)
$$
of $\k G = \bigoplus_{V \in \widehat{G}} \gl(V)$.
\end{defi}

By definition of the simple components $\osp_{\alpha}(V)$
and $\gl_{\alpha}(V)$ one has $\mathcal{L}_{\alpha}(V) \subset
\mathcal{M}_{\alpha}(V)$. It turns out that these
two Lie algebras are the same.

\begin{theo} For any finite group $G$ and multiplicative character
$G \to \k^{\times}$, one has $\mathcal{L}_{\alpha}(G) = \mathcal{M}_{\alpha}(G)$.
In particular, $\mathcal{L}_{\alpha}(G)$ is a reductive Lie
algebra. The center of $\mathcal{L}_{\alpha}(G)$ is generated by
the elements $T_c - \alpha(c) T_{c^{-1}}$ for $c \in \cl(G)$ such that
$\alpha(c) \neq 1$ or $\alpha(c) = 1$ and $c \neq c^{-1}$.
\end{theo}

We give two proofs of this result. The first proof (section 2) uses character
theory and a weighted version of the Frobenius-Schur indicator
in order to show that these two algebras have the same dimension.
The second one (section 3) is more conceptual, in the sense that it proves
the equality of these two objects without any counting argument,
and belongs naturally to the setting of harmonic analysis. Indeed,
when $\k = \C$, the Lie algebra $\mathcal{L}_{\alpha}(G)$ can be identified
with a set of discontinuous measures, defined for any locally compact group $G$,
whereas $\mathcal{M}_{\alpha}(G)$ admits a natural
generalization for compact groups.

One may consider a slightly more general setting, where
the antiautomorphism $S$ has the form $g \mapsto \alpha(g) \tau(g)^{-1}$,
and $\tau$ is an involutive automorphism of $G$.
This defines a Lie subalgebra $\mathcal{L}_{\alpha,\tau}(G)$,
spanned by the elements $g - \alpha(g) \tau(g^{-1})$.
Similarly,
the corresponding permutation of $\widehat{G}$, defined at the level
of characters by $\chi \mapsto \alpha \overline{\chi} \circ \tau$,
has order two, and we can define $\widehat{G}_{\alpha,\tau}^{even}$
and $\widehat{G}_{\alpha,\tau}^{odd}$ accordingly. 
If $\chi$ is an irreducible character of $G$ which corresponds to
$\rho : G \to \GL(V)$,
then $\chi = \alpha \overline{\chi} \circ \tau$
is equivalent to the existence of
a nondegenerate bilinear form $\phi$ on $V$ such that
$\phi(g.x,\tau(g).y) = \alpha(g) \phi(x,y)$, in which case
we get natural Lie subalgebras $\osp_{\alpha,\tau}(V) \subset \gl(V)$
and $\mathcal{M}_{\alpha,\tau}(G) \subset \k G$. Like before,
we check that
$\mathcal{M}_{\alpha,\tau}(G) \supset  \mathcal{L}_{\alpha,\tau}(G)$.
The theorem above can be generalized in $\mathcal{M}_{\alpha,\tau}(G) =  \mathcal{L}_{\alpha,\tau}(G)$.
The second proof we give of the theorem proves this generalization (see corollary \ref{corharm}).

\subsection{Connections with classical and other topics.}

Our primary interest in these Lie algebras was that, when $G$ is finite
group generated by reflections (for example a finite Coxeter group),
with $\eps : G \to \{ \pm 1 \}$ the sign character, then $\mathcal{L}_{\eps}(G)$
contains the Lie subalgebra of $\k G$ generated by the reflections.
It turns out that this latter Lie algebra is closely connected to
the Zariski-closure of the image of the generalized braid group
associated to $G$ inside the corresponding Hecke algebra (see \cite{HECKINF,LIETRANSP,INFIWAHECK}).
We list here a number of other connections.

\subsubsection{Frobenius-Schur theory}

Recall that, if $V \in \widehat{G}$, then the Frobenius-Schur
indicator $\mathcal{F}(V)$ of $V$ is 1 if $V$ can be realized over $\R$, $-1$ if $V$
cannot be realised over $\R$ but has real-valued character, and
$0$ otherwise. It is then said that $V$ has real, quaternionic
or complex type (see e.g. \cite{SERRE}). If $\chi_V$ denotes the character of $V$, then
$$
\mathcal{F}(V) = \frac{1}{\# G} \sum_{g \in G} \chi_V(g^2)
$$
In the first two cases, $1 \hookrightarrow V \otimes V$. In the
first case the bilinear form that we defines is symmetric,
in the second case it is symplectic. In the last case
$1 \not\hookrightarrow V \otimes V$.

It follows that, if $\alpha$ is the trivial character, then the elements
of $\widehat{G}$
whose classes belong to $\widehat{G}_{\alpha}^{odd}$
correspond to the representations
of complex type, whereas the classes belonging to
$\widehat{G}_{\alpha}^{even}$ correspond to representations
of real or quaternionic type. In particular our theorem for $\alpha = 1$
is a Lie-theoretic interpretation of the Frobenius-Schur
theory.

\subsubsection{Real versions of Clifford theory}

Here we assume $\alpha \neq 1$. Then its kernel $H$ is
a normal subgroup of $G$ with cyclic quotient,
and its ordinary representations can be deduced from
those of $G$ through Clifford theory. It is easily noticed, using
character theory and the Frobenius-Schur indicator, that
the real
types of the representations of $G$ and $H$ are intimately related.
Our context provides a somewhat more conceptual explanation of this
phenomenon, by the following elementary fact.

\begin{prop} If $H = \Ker \alpha$,
then $\mathcal{L}(H) = \mathcal{L}(G) \cap \mathcal{L}_{\alpha}(G)$.
\end{prop}
\begin{proof}
For $h \in H$, we have $h - h^{-1} = h - \alpha(h) h^{-1}$
hence $\mathcal{L}(H) \subset \mathcal{L}(G) \cap \mathcal{L}_{\alpha}(G)$.
If $x = \sum \la_g g \in \mathcal{L}(G) \cap \mathcal{L}_{\alpha}(G)$,
then $\la_g = -\la_{g^{-1}}$ and $\la_g = - \la_{g^{-1}}$
for all $g \in G$, hence $\la_g = 0$ or $\alpha(g) = 1$,
that is $g \not\in H \Rightarrow \la_g = 0$.
Hence $x \in \k H \cap \mathcal{L}(G) = \mathcal{L}(H)$.

\end{proof}

\subsubsection{The Kawanaka-Matsuyama indicator}

In \cite{KAWA} was introduced an indicator $c_{\tau}(\chi)$
for $\chi$ a character of $G$ and $\tau \in \Aut(G)$
with $\tau^2 = 1$, defined by
$$
c_{\tau}(\chi) = \frac{1}{\# G} \sum_{g \in G} \chi(g \tau(g))
$$
When $\tau = 1$ we have $c_{\tau} = \mathcal{F}$, hence $c_{\tau}$
is a twisted version of the Frobenius-Schur indicator. If $\chi$
is irreducible, then
$c_{\tau}(\chi) = \{ -1 , 0 , 1\}$, with $c_{\tau}(\chi) \geq 0$
iff $\chi(\tau(g)) = \chi(g^{-1})$, and $c_{\tau}(\chi) = 1$
if and only if the corresponding representation admits
a model $R : G \to \GL_N(\C)$ such that $R(\tau(g)) = \overline{R(g)}$
for all $g \in G$.

Here we introduce a weighting $\mathcal{F}_{\alpha}$
of the Frobenius-Schur indicator, which
is connected with our
Lie algebra structures:
$$
\mathcal{F}_{\alpha}(\chi) = \frac{1}{\# G} \sum_{g \in G} \chi(g^2)
\overline{\alpha(g)}.
$$
It is easily checked that,
if $L = G \rtimes <\tau >$ and $\eps : L \to \{ \pm 1 \}$
has kernel $G$, then
$2 \mathcal{F}_{\eps}(\chi) = \mathcal{F}_1(\Res_G \chi) - c_{\tau}(
\Res_G \chi)$.
Moreover, by Clifford theory, we know that if $\Res_G \chi$
is not irreducible, then $\Res_G \chi = \chi_+ + \chi_-$
with $\chi_{\pm}$ irreducible and $\chi_{\pm} \circ \tau
= \chi_{\mp}$. It follows that $\chi_{\pm}(g \tau(g)) = \chi_{\pm}(\tau(g) g)=
\chi_{\mp}(g \tau(g))$ hence $c_{\tau}(\chi_{\pm}) = c_{\tau}(\chi_{\mp})$.
Likewise,
$$\mathcal{F}(\chi_{\pm}) = \frac{1}{\# G}\sum\chi_{\pm}(g^2) 
\frac{1}{\# G}\sum\chi_{\mp}(\tau(g)^2) = \frac{1}{\# G}\sum\chi_{\mp}(g^2)=
\mathcal{F}(\chi_{\mp}).$$
It follows that 
$\mathcal{F}_{\eps}(\chi) =
\mathcal{F}_1(\chi_{\pm}) - c_{\tau}(
\chi_{\pm})$, hence these twisted Frobenius-Schur indicators
are closely related to our weighted ones.

Finally, note that our twisted Lie algebras $\mathcal{L}_{1,\tau}(G)$
and $\mathcal{M}_{1,\tau}(G)$ provide a Lie-theoretic interpretation
of the Kawanaka-Matsuyama indicator (see remark 3.18 in \cite{KAWA}).
 
\subsubsection{Bessel functions}

Assume that $\k$ is a complete topological field.
The structure of $\mathcal{M}_{\alpha}(G)$ and
the isomorphism $\mathcal{M}_{\alpha}(G) \simeq \mathcal{L}_{\alpha}(G)$
provides a description and a decomposition of the Lie group $\exp \mathcal{L}_{\alpha}(G)$.
On the other hand, we remark here that a direct exponentiation
of $\mathcal{L}_{\alpha}(G)$, when $\k = \C$, involves Bessel's $J$
function.

Recall that $J_m(z)$ for $m \in \Z$ can
be defined by
$$
J_m(z) = \frac{1}{2 \pi} \int_{-\pi}^{\pi} e^{\ii z \sin t} e^{-\ii m t} dt
= \sum_{k=0}^{\infty} \frac{(-1)^k}{\Gamma(k+m+1) k !} \left(
\frac{z}{2} \right)^{2k+m}
$$
Also, for all $z,q \in \C$ with $|q| = 1$, we have by Fourier expansion
$$
\exp \frac{z}{2} (q - q^{-1}) = \sum_{m = -\infty}^{+ \infty}
J_m(z) q^m
$$
In section 3, we generalize the construction of $\mathcal{L}_{\alpha}(G)$
to the case of a locally compact group. In this setting,
$\mathcal{L}_{\alpha}(G)$ is a subspace of the Banach algebra of
totally discontinuous measures on $G$, and $x - \alpha(x)x^{-1}$
corresponds to the measure $\delta_x - \alpha(x)\delta_{x^{-1}}$, with
$\delta_x$ the Dirac measure. When $G$ is infinite
cyclic generated by $x$, $\alpha : G \to \C^{\times}$ sending $x$ to $\omega$
of modulus 1, the above formula translates as
$$
\exp \frac{z}{2} (\delta_{x} - \om \delta_{x^{-1}}) = \sum_{m = -\infty}^{+ \infty}
J_m(z \varphi) \varphi^{-m}\delta_{x^m}
$$
where $\varphi \in \C^{\times}$ with $\varphi^2 = \omega$. Note that
$J_m(z \varphi) \varphi^{-m}$ is well-defined, since $J_m(-z) = (-1)^m J_m(z)$.
Assume that $\alpha(G)$ is finite, $H < \Ker \alpha$ and $Q = G/H \simeq \mu_N(\C)$.
Then the canonical morphism $\mathcal{L}_{\alpha}(G) \to \mathcal{L}_{\alpha}(Q)$
is continuous. Letting $y$ denote the image of $x$, we thus get the
formula
$$
\exp(\frac{z}{2} (y-\om y^{-1}) = \sum_{r=0}^{N-1}
y^r \left( \sum_{m \equiv r \mod N} J_m(z \varphi) \varphi^{-m} \right).
$$
which describes the commutative Lie group $\exp \mathcal{L}_{\alpha}(Q)$
inside $(\C Q)^{\times}$ when $Q$ is a finite cyclic group.

\medskip
\noindent {\bf Acknowledgements.} It is my pleasure to thank J. Vargas
and O. Glass for useful discussions and references about harmonic
analysis.

\section{First proof, through character theory}

\subsection{Weighted Frobenius-Schur indicator}

We define the weighted Frobenius-Schur indicator associated to $\alpha$
as the additive function on the representation ring of $G$ defined
by
$$
\mathcal{F}_{\alpha}(V) = \frac{1}{\# G} \sum_{g \in G} \chi_V(g^2)
\overline{\alpha(g)}.
$$
If $\alpha = 1$ one recovers the usual Frobenius-Schur indicator. Its
weighted version has similar properties. Recall for instance that
the number of involutions in a finite group $G$ is $\sum_{V \in \widehat{G}} \mathcal{F}_{1} (V) \dim V$.
Here we introduce
$$
\mathcal{I}_{\alpha}(G) = \{ g \in G  \ | \ g = g^{-1},\ \ \alpha(g) = 1 \}
\ \ \ 
\mathcal{J}_{\alpha}(G) = \{ g \in G  \ | \ g = g^{-1},\ \ \alpha(g) = -1 \}
$$
and we denote $\mathrm{reg}$ the regular representation of $G$.

\begin{prop} \label{proptwschur} If $V \in \widehat{G}$ then $\mathcal{F}_{\alpha}(V) \in \{ -1,0,1 \}$.
In these three cases the location of $\alpha$ in the decomposition of
$V \otimes V$ is as follows
$$
\begin{array}{|c||c|c|c|}
\hline
\mathcal{F}_{\alpha}(V) & -1 & 1 & 0 \\
\hline
\Hom(\alpha,V\otimes V) & \alpha \into \Lambda^2 V  &
\alpha \into S^2 V & \alpha \not\into V \otimes V \\
\hline
\end{array}
$$
Moreover,
$$
\mathcal{F}_{\alpha}(\mathrm{reg}) = \sum_{V \in \widehat{G}} \mathcal{F}_{\alpha}(V) \dim V = \mathcal{I}_{\alpha}(G)
- \mathcal{J}_{\alpha}(G)
$$
\end{prop}
\begin{proof}
The classical formulas $\chi_{S^2 V} (g) = \frac{1}{2} (\chi_V(g)^2 +
\chi_V(g^2))$ and $\chi_{\Lambda^2 V} (g) = \frac{1}{2} (\chi_V(g)^2 -
\chi_V(g^2))$ imply that $\mathcal{F}_{\alpha}(V) = (\chi_{S^2 V}| \alpha)
- (\chi_{\Lambda^2 V} | \alpha)$, where $(\ | \ )$ denotes the
usual scalar product on the class functions on $G$. If $V$ is irreducible, then
$V^* \otimes \alpha$ is also irreducible and $\Hom(\alpha,V \otimes V)
\simeq  \Hom(V^* \otimes \alpha,V)$ has dimension at most 1 by the Schur
lemma. This establishes the first part of the lemma.

One has $\chi_{\mathrm{reg}}(g^2) = 0$ if $g \neq g^{-1}$. Otherwise
$g^2 = e$ hence $\alpha(g)^2 = 1$ and $\alpha(g) \in \{-1,1\}$. It follows
that $\mathcal{F}_{\alpha}(\mathrm{reg}) = \mathcal{I}_{\alpha}(g) -
\mathcal{J}_{\alpha}(g)$.
\end{proof}

\subsection{Identification of the Lie algebras.}

Since $\sum_{V \in \widehat{G}} \dim \gl(V) = \# G$ one gets
$$
\# G - 2 \dim \mathcal{M}_{\alpha}(G) = \sum_{\alpha \into S^2 V} \dim V
- \sum_{\alpha \into \Lambda^2 V} \dim V = \sum_{V \in \widehat{G}}
\mathcal{F}_{\alpha}(V) \dim V = \mathcal{F}_{\alpha}(\mathrm{reg})
$$
On the other hand, $\mathcal{L}_{\alpha}(G)$ is spanned
by the elements $g - \alpha(g) g^{-1}$ such that $g \neq g^{-1}$
or $g = g^{-1}$ but $\alpha(g) \neq 1$. It follows that
$\# G - 2 \dim \mathcal{L}_{\alpha}$ equals
$$
\begin{array}{ll}
 & \# \{ g \in G \ | \ g = g^{-1} \}
- 2 \# \{ g \in G \ | \ g = g^{-1}, \ \alpha(g) \neq 1 \} \\
= & \# \{g \in G \ | \ g=g^{-1}, \ \alpha(g) = 1 \} - 
\# \{ g \in G \ | \ g = g^{-1}, \ \alpha(g) \neq 1 \} \\
= & \# \{g \in G \ | \ g=g^{-1}, \ \alpha(g) = 1 \} - 
\# \{ g \in G \ | \ g = g^{-1}, \ \alpha(g) = - 1 \}
\end{array}
$$
since $g = g^{-1}$ implies $\alpha(g) \in \{ -1,1\}$. It follows
that $\# G - 2 \dim \mathcal{L}_{\alpha}(G) = \mathcal{I}_{\alpha}(G)
- \mathcal{J}_{\alpha}(G) = \# G - 2 \dim \mathcal{M}_{\alpha}(G)$
by proposition \ref{proptwschur}, hence $\mathcal{L}_{\alpha}(G) =  
\mathcal{M}_{\alpha}(G)$.

\subsection{Description of the center.}

From the identification above one gets that the center
of $\mathcal{L}_{\alpha}(G)$ as dimension $\frac{1}{2} \{ V \in \widehat{G} \ | \ 
V \not\simeq \alpha \otimes V^* \}$. Let $\cl(G)$ denote the set
of conjugacy classes of $G$, and recall the notation $T_c$ for $c \in
\cl(G)$ from \S 1. We consider the elements
$T_c - \alpha(c) T_{c^{-1}}$ for $c \in \cl(G)$ such that $\alpha(c) \neq 1$
or $\alpha(c) = 1$ and $c \neq c^{-1}$. They span a subspace of $Z(\k G)$
of dimension
$$
\frac{1}{2} \# \cl(G) - \frac{1}{2} \# \{ c \in \cl(G) \ | \ \alpha(c) = 1,
\ c = c^{-1} \}
$$
\begin{lemma} Let $G$ be a finite group, and $\alpha : G \to \kt$ be
a multiplicative character. Then 
$$
\# \{ c \in \cl(G) \ | \ \alpha(c) = 1,
\ c = c^{-1} \} = \{ V \in \widehat{G} \ | \ V \simeq \alpha \otimes V^* \}
$$
\end{lemma}

\begin{proof}

The space of central functions over $G$ with values in $\k$ has two
natural basis, given by the irreducible characters on the
one hand, and by the characteristic functions $\varphi_c$ for
$c \in \cl(G)$ on the other hand. We define an involutory
endomorphism $S$ of this space by $S(\varphi)(g) = \alpha(g) \varphi(g^{-1})$.

The functions $\chi_V$ for $V \in \widehat{G}, \ \chi_V = \alpha \otimes \overline{\chi_V}$
and the functions $\chi_V + \alpha \overline{\chi_V}$ for $\chi_V \neq \alpha \otimes \overline{\chi_V}$
form a basis of $\Ker(S-1)$. Another basis is given by the
$\varphi_c$ for $\alpha(c) = 1$, $c = c^{-1}$ and the $\varphi_c + \alpha
\varphi_{c^{-1}}$ for $\alpha(c) \neq 1$ or $c \neq c^{-1}$.

In the same way, two basis of $\Ker(S+1)$ are given by
$\{ \chi_V - \alpha \overline{\chi_V} \ |\  V \not\simeq \alpha \otimes V^* \}$
and $\{ \varphi_c - \alpha \varphi_{c^{-1}} \ | \ \alpha(c) \neq 1 \mbox{ or }
c \neq c^{-1} \}$. It follows that
$$
\begin{array}{ll}
 & \# \{ c \in \cl(G) | \alpha(c) =1, c = c^{-1} \} \\
= & \dim \Ker(S-1) - \dim \Ker(S+1) \\
 = & \# \{ V \in \widehat{G} \ | \ V \simeq \alpha \otimes V^* \}
\end{array}
$$
\end{proof}

Since $\# \cl(G) = \# \widehat{G}$ this lemma concludes the proof of the theorem.

\section{Second proof, through harmonic analysis}

We will use \cite{HR} as our main reference here. In order
to prove the theorem, since we already know $\mathcal{L}_{\alpha,\tau}(G)
\subset \mathcal{M}_{\alpha,\tau}(G)$, we can assume $\k = \C$
without loss of generality.

\subsection{Preliminaries on borelian measures}

\subsubsection{Weak-* topology}

Let $X$ be a locally compact topological space. We let $C_0(X)$
be the $\C$-vector space of functions on $G$ that tend to 0 at infinity,
and $\mathbf{M}(X)$ the $\C$-vector space of bounded complex borelian
measures on $X$, that we identify to bounded linear forms on $C_0(X)$.

The space $\mathbf{M}(X)$ admits two topologies which are useful here, the
norm topology defined by the operator norm, and the weak-* topology
for which a basis of open sets is given by
$$
U(\Phi ; f_1,\dots,f_k,\eps) = \{ \Psi \in \mathbf{M}(X) \ | \ 
\forall i \in [1,k] \ | \Psi(f_i) - \Phi(f_i) | < \eps \}
$$
where $\Phi \in \mathbf{M}(X)$, $f_1,\dots,f_k \in C_0(X)$ and
$\eps > 0$.

We will make repeated use of the following basic lemma.

\begin{lemma}\label{lemweak} Let $E$ be a subspace of $\mathbf{M}(X)$ such that,
for all $\varphi \in C_0(X)$, one has
$$
\forall \mu \in E \ \ \mu(f) = 0 \Rightarrow f = 0
$$
Then, for all $\la \in \mathbf{M}(X)$ and $f_1,\dots,f_n \in C_0(X)$,
there exists $\mu \in E$ such that $\forall i \in [1,n]$ one has
$\mu(f_i) = \la(f_i)$. In particular $\forall \eps > 0 \ \ \mu
\in U(\la; f_1,\dots,f_n, \eps)$ and $E$ is dense in $\mathbf{M}(X)$
for the weak-* topology.
\end{lemma}
\begin{proof} Without loss of generality we assume that the family $f_1,\dots,f_n$
is linearly independant. It follows that
the linear map $E \to \C^n$ given by $\mu \mapsto (\mu(f_i))_{i=1\dots
n}$ is surjective, and in particular there exists $\mu \in E$
such that $\mu(f_i) = \la(f_i)$ for all $i \in [1,n]$.
\end{proof}

We recall from \cite{HR} the definition of the following subspaces
$$
\begin{array}{lcl}
\mathbf{M}_d(X) & =& \{ \mu \in \mathbf{M}(X) \ | \ \exists E \subset
X \mbox{ countable } \ |\mu|(X \setminus E) = 0 \} \\
\mathbf{M}_c(X) & =& \{ \mu \in \mathbf{M}(X) \ | \ \forall x \in X \ \ 
\mu(\{ x \}) = 0 \} \\
\end{array}
$$
The subspaces $\mathbf{M}_d(X)$ and $\mathbf{M}_c(X)$ contain the purely discontinuous
and continuous measures, respectively. For the norm topology,
they are closed in $\mathbf{M}(X)$.

For any $x \in X$
we define the punctual measure $\delta_x \in \mathbf{M}_d(X)$
by $\delta_x(A) = 1$ if $x \in A$, $\delta_x(A) = 0$ otherwise.
Elements of $\mathbf{M}_d(X)$ have the form $\sum a_n \delta_{x_n}$
where $x_n \in X$ and $\sum | a_n | < \infty$.

\subsubsection{Haar measure and convolution}

Let $G$ be a locally compact topological group with neutral element $e$, and $d x$ a left-invariant
Haar measure on $G$. We let $\mathbf{M}_a(G)$ be the subspace
in $\mathbf{M}_c(G)$ of measures which are absolutely
continuous with respect to $d x$.

The multiplication $G \times G \to G$ gives rise to the convolution
of measures $(\mu_1,\mu_2) \mapsto \mu_1 * \mu_2$. With this operation
(see \cite{HR} 19.6) $\mathbf{M}(G)$ is a Banach algebra whose unit
is $\delta_e$. Then $\mathbf{M}_d(G)$ is a subalgebra of $\mathbf{M}(G)$
and $\mathbf{M}_c(G)$ as well as its subspace $\mathbf{M}_a(G)$ are
ideals on both side for $\mathbf{M}(G)$. All of them are closed
with respect to the norm topology. We refer to \cite{HR} 19.15, 19.16 and 19.18 for
all these elementary facts.

\subsubsection{What happens when $G$ is finite ? }

If $G$ is compact, $d x$ can be chosen
such that $d x(G) = 1$, and $d x$ is also right-invariant. Moreover,
for all $f \in L^1(G)$ we have  (see \cite{HR} 20.2 (ii))
$$
\int_G f(t^{-1}) dt = \int_G f(t) dt
$$
Finally, recall that we can identify $L^1(G)$ with $\mathbf{M}_a(G)$ through $\varphi \mapsto \int f \varphi$.

If $G$ is discrete, one has $\mathbf{M}_d(G) = \mathbf{M}_c(G) =
\mathbf{M}(G)$. If $G$ is compact and discrete, that is if $G$ is finite
with the discrete topology, then
one has moreover $\mathbf{M}(G) = \mathbf{M}_a(G) = L^1(G)$, so that
all the algebras introduced here are identified. Moreover,
the group algebra $\C G$ can be identified with them by $g \mapsto \delta_g$.

\subsection{The Lie algebras $\mathcal{L}_{\alpha}(G)$ and
$\mathbf{M}_{\alpha}(G)$}

Let $G$ be a locally compact topological group, $\alpha : G \to \C^{\times}$
a continuous bounded character, and $\tau$ a continuous bounded
automorphism of $G$ with $\tau^2 = 1$. To any $\varphi \in C_0(G)$ we
associate $\hat{\varphi} \in C_0(G)$ defined by
$\hat{\varphi}(x) = \alpha(x) \varphi(\tau(x)^{-1})$. This is a linear
endomorphism of $C_0(G)$, from which we deduce an endomorphism
$\mu \mapsto \mu^+$
of $\mathbf{M}(G)$. We have $\mu^+(\varphi) = \hat{\varphi}$ for
all $\varphi \in C_0(G)$.

\begin{prop} The linear map $\mu \mapsto \mu^+$ is an involutive
antiautomorphism of algebra of $\mathbf{M}(G)$ setwise stabilizing $\mathbf{M}_d(G)$.
 It is continuous for the norm topology and for the weak-* topology, and
$\delta_x^+ = \alpha(x) \delta_{\tau(x)^{-1}}$ for all $x \in G$.
\end{prop}
\begin{proof} For $\psi \in C_0(G\times G)$
we denote $\hat{\psi}$ the function $\psi(x,t) = \alpha(xt) \psi(x^{-1},
t^{-1})$. Let $\mu_1,\mu_2 \in \mathbf{M}(G)$.
It is easily checked that $\mu_1^+ \otimes \mu_2^+(\psi) = 
\mu_1 \otimes \mu_2(\hat{\psi})$ for all $\psi \in C_0(G) \otimes C_0(G)$,
hence, for all $\psi \in C_0(G \times G)$ by density of $C_0(G) \otimes
C_0(G)$ in $C_0(G \times G)$. The verification that $\mu \mapsto \mu^+$ is 
an antiautomorphism of algebras follows by an easy calculation.
It is involutive because $\varphi \mapsto \hat{\varphi}$ is involutive for
$\varphi \in C_0(G)$. It is continuous for the norm topology because it
is $|| \alpha ||_{\infty}$-lipschitz, and for the weak-* topology because
the inverse image of $U(\la; f_1,\dots,f_k,\eps)$ is
the open set $U(\la^+; \hat{f}_1,\dots,\hat{f}_k,\eps)$. Finally, for
all $\varphi\in C_0(G)$ one has $\delta_x^+(\varphi) = \alpha(x) \delta_{x^{-1}}(\varphi)
$ hence $\delta_x^+ = \alpha(x) \delta_{x^{-1}}$. It follows that the set
of linear combinations of punctual measures is stable under $\mu \mapsto
\mu^+$, and so is its closure $\mathbf{M}_d(G)$ for the norm
topology by a continuity argument.
\end{proof}

\begin{cor} The vector space $\mathbf{M}_{\alpha}(G) =
\{ \mu \in \mathbf{M}(G) \ | \ \mu^+ = -\mu \}$
is a Lie subalgebra of $\mathbf{M}(G)$ which is closed for the
norm and weak-* topology.
\end{cor}

\begin{defi} We let $\mathcal{L}_{\alpha}(G)$ be the Lie subalgebra
of $\mathbf{M}_d(G)$ spanned by the elements $\delta_x - \delta_x^+ = \delta_x - \alpha(x) \delta_{\tau(x)^{-1}}$
for $x \in G$.
\end{defi}

\begin{prop} \label{propLadense} $\mathcal{L}_{\alpha}(G)$ is dense in $\mathbf{M}_d(G)
\cap \mathbf{M}_{\alpha}(G)$ for the norm topology, and is
dense in $\mathbf{M}_{\alpha}(G)$ for the weak-* topology.
\end{prop}
\begin{proof}
By definition $\mathbf{M}_d(G) \cap \mathbf{M}_{\alpha}(G) =
\{ \mu \in \mathbf{M}_d(G) \ | \ \mu^+ = - \mu \}$ contains
$\mathcal{L}_{\alpha}(G)$. An element of $\mathbf{M}_d(G)$
can be written as $\mu = \sum a_n \delta_{x_n}$ with $x_n \in X$
being distincts and $\sum | a_n | < \infty$. Because of
the continuity of $\mu \mapsto \mu^+$ for the norm topology,
the condition $\mu^+ = - \mu$ implies $\mu = \sum b_n (\delta_{y_n}
- \delta_{y_n}^+) $ where $\{ y_n \} \cup \{ y_n^{-1} \} = \{ x_n \}$,
hence $\mu$ is the limit of a sequence of elements in $\mathcal{L}_{\alpha}
(G)$ with respect to the norm topology. Since $\mathbf{M}_d(G) \cap
\mathbf{M}_{\alpha}(G)$ is closed in $\mathbf{M}(G)$ for the norm
topology in follows that $\mathcal{L}_{\alpha}$ is dense in
$\mathbf{M}_d(G) \cap \mathbf{M}_{\alpha}(G)$ for this topology.

Let $E$ be the vector subspace of $\mathbf{M}_d(G)$ spanned by the
$\delta_x, x \in G$. Because of $\forall x \in G \ \ \delta_x(\varphi) = 0
\Rightarrow \varphi = 0$ lemma \ref{lemweak} implies that $E$ is dense
in $\mathbf{M}(G)$ for the weak-* topology. Moreover, if $\la \in
\mathbf{M}_{\alpha}(G)$ and a neighborhood
$U = U(\la ; f_1,\dots,f_k,\eps)$ of $\la$ are given, we let
$V = U(\la ; f_1,\dots,f_k, \hat{f}_1,\dots,\hat{f}_k,\eps) \subset U$
be a smaller open neighborhood of $\la$. Again because of lemma
\ref{lemweak} there exists $\mu \in E$ such that, for all $i \in [1,k]$,
$\mu(f_i) = \la(f_i)$ and $\mu(\hat{f}_i) = \la(\hat{f}_i)$. Then $\nu =
\frac{\mu - \mu^+}{2} \in \mathcal{L}_{\alpha}(G)$ satisfies $\nu \in U$.
Because $\mathbf{M}_{\alpha}(G)$ is closed in $\mathbf{M}(G)$
for the weak-* topology is follows that $\mathcal{L}_{\alpha}(G)$
is dense in $\mathbf{M}_{\alpha}(G)$.
\end{proof}

\begin{cor} If $G$ is finite then $\mathcal{L}_{\alpha}(G) = \mathbf{M}_{\alpha}(G)$.
\end{cor}

\subsection{The Lie algebra $\mathcal{M}_{\alpha}(G)$ of a
compact group}

Let $G$ be a \emph{compact} topological group, and $\alpha : G \to \C^{\times}$
a (necessarily bounded) continuous character. We have $|\alpha(x)| = 1$
for all $x \in G$. Let $\tau$ a continuous involutive automorphism
of $G$ which preserves the Haar measure. If $C(G)$ denotes the
space of (complex valued) functions on $G$, we have $C(G) = C_0(G)
\subset L^1(G) \simeq \mathbf{M}_a(G)$. We let
$N_f \in \mathbf{M}_a(G)$ be the measure corresponding to $f \in L^1(G)$
with respect to the Haar mesure of volume 1 chosen on $G$. For
$f \in L^1(G)$ we let $f^{\bigstar} \in L^1(G)$ denote $x \mapsto f(x^{-1})$.
Notice that, if $f \in C(G)$ then $f^{\bigstar} \in C(g)$ and that,
for all $f \in L^1(G)$ we have $\int f = \int f^{\bigstar}$.

\begin{lemma} $\mathbf{M}_a(G)$ is stable under $\mu \mapsto \mu^+$
and, for all $f \in L^1(G)$, $N_f^+ = N_{(\alpha f)^{\bigstar}\circ \tau}$
\end{lemma}
\begin{proof} Since $f \in L^1(G) \Rightarrow \alpha f \in L^1(G)$
it is sufficient to prove the latter statement. Let $f \in L^1(G)$ and $\varphi
\in C(G)$ Then
$$
N_f^+(\varphi) = N_f(\hat{\varphi}) = \int \alpha (\varphi^{\bigstar} \circ \tau) f
= \int \alpha^{\bigstar} (\varphi\circ \tau) f ^{\bigstar} = \int \varphi
(\alpha f)^{\bigstar}\circ \tau = N_{(\alpha f)^{\bigstar}\circ \tau}(\varphi)
$$
hence $N_f^+ = N_{(\alpha f)^{\bigstar}\circ \tau} \in \mathbf{M}_a(G)$.
\end{proof}

From now on we identify $L^1(G)$ with $\mathbf{M}_a(G)$, hence $C(G)$ with a subspace
of $\mathbf{M}_a(G)$. In particular we note $f^+ = (\alpha f)^{\bigstar}\circ \tau$.

\begin{defi} Let $\mathcal{M}_{\alpha}^1(G) = \{ \mu \in \mathbf{M}^a(G)
\ | \ \mu^+ = -\mu \} = \mathbf{M}_a(G) \cap \mathbf{M}_{\alpha}(G)$
and $\mathcal{M}_{\alpha}^c(G) = C(G) \cap \mathcal{M}_{\alpha}^1(G) =
C(G) \cap \mathbf{M}_{\alpha}(G)$.
\end{defi}

\begin{prop} \label{propMaCdense} $\mathcal{M}_{\alpha}^1(G)$ and $\mathcal{M}_{\alpha}^c(G)$
are dense in $\mathbf{M}_{\alpha}(G)$ for the weak-* topology.
\end{prop}
\begin{proof} The proof is similar to the one in proposition
\ref{propLadense}. Let $\la \in  \mathbf{M}_{\alpha}(G)$,
$U = U(\la;f_1,\dots,f_n,\eps)$ an open neighborhood
of $\la$ and $V = U(\la;f_1,\dots,f_n,\hat{f}_1,\dots,\hat{f}_n,\eps) \subset U$.
Let $\varphi \in C(G)$. One has
$$
\begin{array}{rcl}
(\forall f \in L^1(G) \ \ \int f \varphi = 0) & \Rightarrow & \varphi = 0 \\
(\forall f \in C(G) \ \ \int f \varphi = 0) & \Rightarrow & \varphi = 0
\end{array}
$$
hence by lemma \ref{lemweak} there exists $\mu = N_f$ with $f \in L^1(G)$
or $f \in C(G)$ such that, for all $k \in [1,n]$ one has
$\mu(f_k) = \la(f_k)$ and $\mu(\hat{f}_k) = \la(\hat{f}_k)$. Letting
$\nu = \frac{\mu - \mu^+}{2}$ one gets $\nu \in \mathcal{M}_{\alpha}^1(G)$
or $\nu \in \mathcal{M}_{\alpha}^c(G)$ and $\nu \in V$.
\end{proof}

Let $\widehat{G}$ denote the unitary dual of $G$, namely the set of irreducible
unitary representations $u : G \to U(H)$ up to isomorphism.
Since $G$ is compact, all such Hilbert spaces $H$ endowed with their
$G$-invariant hermitian scalar product $<\ , \ >_H$
are finite dimensional. We let $\mathcal{A}_u(G) \subset C(G)$ be
the space spanned by the matrix coefficients $a^u_{v,w} : x \mapsto
<u(x) v,w>_H$ for $v,w \in H$. It is a classical fact (see \cite{HR2} \S 27 (27.49))
that this correspondence induces an isomorphism of left $G$-modules
$\mathcal{A}_u(G)
\simeq H^* \otimes H \simeq \End(H)^* \simeq \End(H^*)$. We let
$\mathcal{A}(G) \subset C(G)$ be the space spanned by the $\mathcal{A}_u(G)$
for $u \in \widehat{G}$. It is the direct sum of these
subspaces, and it is a subalgebra of $C(G)$ under the convolution
product.

\begin{prop} For all $u \in \widehat{G}$ we have $\mathcal{A}_u(G)^+
= \mathcal{A}_{(\alpha u)^*\circ \tau} (G)$ and, in particular,
$\mathcal{A}(G)$ is stable under $f \mapsto f^+$, as well as
$\mathcal{A}_u(G)$ if $u \simeq (\alpha u)^*$. Moreover,
if $u \not\simeq (\alpha u)^*$ then the linear map
$f \mapsto f - f^+$ is a Lie algebra isomorphism from $\mathcal{A}_u(G)$ to
$(\mathcal{A}_u(G) \oplus \mathcal{A}_u(G)^+) \cap \mathbf{M}_{\alpha}(G)$.
\end{prop}
\begin{proof}
The first statement comes from the easily checked formula
$(\alpha a_{v,w}^u)^{\bigstar} = a^{(\alpha u)^*}_{w^*,v^*}$,
where $w^*$ for $w \in H$ denotes the linear form $<\ , w>$,
with $<\ ,\ >$ the given hermitian product on $H$. If
$u \not\simeq (\alpha u)^*\circ \tau$ then $\mathcal{A}_u(G)$
and $\mathcal{A}_u(G)^+$ are, as algebras under the convolution product,
distinct direct summands of $\mathcal{A}(G)$. Since $f \mapsto f^+$
is an antiautomorphism of the algebra $\mathcal{A}(G)$ it follows that
$f \mapsto f-f^+$ is a Lie algebra morphism $\mathcal{A}_u(G)
\to \mathcal{A}_u(G) \oplus \mathcal{A}_u(G)^+$ whose image
is contained in $\mathbf{M}_{\alpha}(g)$. Injectivity now comes
from the fact that $\mathcal{A}_u(G) \cap \mathcal{A}_u(G)^+
= \mathcal{A}_u(G) \cap \mathcal{A}_{(\alpha u)^*}(G) = \{ 0 \}$
and surjectivity from the fact that, if $f_1 \in \mathcal{A}_u(G)$
and $f_2 \in \mathcal{A}_u(G)^+$ satisfy $f_1 + f_2 \in
\mathbf{M}_{\alpha}(G)$, then $f_1^+ + f_2^+ = -f_1 - f_2$ hence
$f_1^+ + f_2 = -f_1 - f_2^+ \in \mathcal{A}_u(G) \cap
\mathcal{A}_u(G)^+ = \{ 0 \}$. It follows that $f_2 = - f_1^+$ and
$f_1 + f_2 = f_1^+ - f_1^+$ indeed belongs to the image of
$\mathcal{A}_u(G)$.
\end{proof}

As in the finite group case we introduce the equivalence relation
generated by $u  \sim (u^* \otimes \alpha)\circ \tau$ on $\widehat{G}$, and
$\widehat{G}_{\alpha} = \widehat{G}_{\alpha}^{even} \sqcup \widehat{G}_{\alpha}^{odd}$.
Note that $\mathcal{A}_{((u^* \otimes \alpha)\circ \tau)^*}(G) = \mathcal{A}_{(u \otimes \alpha^*)\circ \tau}(G) 
= \mathcal{A}_{u^*}(G)^+$.
If $\tilde{u} \in \widehat{G}_{\alpha}$ has cardinality 2, that
is $\tilde{u} = \{ u , (u^* \otimes \alpha)\circ \tau \} \in \widehat{G}_{\alpha}^{odd}$, then
we let $\gl'_{\alpha}(\tilde{u}) = (\mathcal{A}_{u^*}(G) \oplus
\mathcal{A}_{u^*}(G)^+) \cap \mathbf{M}_{\alpha}(G) \simeq \mathcal{A}_{u^*}(G)$. If
$\tilde{u} = \{ u \} \in \widehat{G}_{\alpha}^{even}$ then we let
$\osp'_{\alpha}(\tilde{u} ) = \{ f \in \mathcal{A}_{u^*}(G) \ | \ f^+ = -f \}$.

\begin{defi}
Let $\mathcal{M}_{\alpha}^{\circ}(G) = \mathcal{A}(G) \cap \mathbf{M}_{\alpha}(G)$,
that is
$$
\mathcal{M}_{\alpha}^{\circ}(G) = \left( \bigoplus_{\tilde{u} \in \widehat{G}_{\alpha}^{odd}}
\gl_{\alpha}(\tilde{u}) \right)
\oplus \left( \bigoplus_{\tilde{u} \in \widehat{G}_{\alpha}^{even}}
\osp_{\alpha}(\tilde{u}) \right)
$$
\end{defi}

Recall that the direct sums involved here are orthogonal ones with
respect to the usual hermitian scalar product $(f|g) = \int f \overline{g}$
on $C(G) \subset L^2(G)$. A straightforward calculation shows that
$f \mapsto f^+$ is unitary with respect to this scalar product :
$$
(f^+ | g^+ ) = \int \alpha^{\bigstar} f ^{\bigstar} \overline{\alpha
^{\bigstar}} \overline{g ^{\bigstar}} = 
\int \alpha \overline{\alpha} f ^{\bigstar} \overline{g}^{\bigstar}
= \int (f \overline{g}) ^{\bigstar} = \int f \overline{g} = (f| g)
$$
because $\alpha ^{\bigstar} = \overline{\alpha} = \alpha^{-1}$.
We then have
\begin{prop} $\mathcal{M}^{\circ}_{\alpha}(G)$ is dense in $\mathbf{M}_{\alpha}(G)$
for the weak-* topology.
\end{prop}
\begin{proof} Because of proposition \ref{propMaCdense} it is sufficient
to prove that $\mathcal{M}_{\alpha}(G)$ is dense in $\mathcal{M}_{\alpha}^c(G)$.
Recall that $\mathcal{A}(G)$ is dense in $C(G)$ for the $L^2$-topology,
that is the norm topology associated to the usual hermitian scalar
product defined above. It follows that $\mathcal{M}_{\alpha}(G) =
\mathcal{A}(G) \cap \mathbf{M}_{\alpha}$ is dense
in $\mathcal{M}_{\alpha}^c(G) = C(G) \cap \mathbf{M}_{\alpha}$
for the same topology. Indeed, for all $f \in \mathcal{M}_{\alpha}^c(G)
\subset C(G)$ and $\eps > 0$, taking $g \in \mathcal{A}(G)$ such that
$ \| f - g \|_{L^2} \leq \eps$ one gets $\| f - \tilde{g} \|_{L^2}
\leq \eps$ for $\tilde{g} = \frac{g - g^+}{2}$ by unitarity of $f \mapsto
f^+$. Now, if $\la = N_f \in U = U(\la; \varphi_1,\dots,\varphi_n , \eps)$
and $\la \in \mathcal{M}_{\alpha}^c(G)$, that is $f \in C(G)$
and $f^+ = -f$, let $m = \max \| \varphi_k \|_{L^2}$ and
$\mu = N_g \in \mathcal{M}_{\alpha}(G)$ such that $\| f - g \|_{L^2}
\leq \frac{\eps}{m}$. One gets, for all $k \in [1,n]$,
$$
| \la(\varphi_k) - \mu(\varphi_k) | = \left| \int (f-g)\varphi_k
\right|  \leqslant \| f - g \|_{L^2} \| \varphi_k \|_{L^2}
\leqslant \eps
$$
by the Cauchy-Schwartz inequality. It follows that $\mu \in U$ hence
$\mathcal{M}_{\alpha}(G)$ is dense in $\mathbf{M}_{\alpha}(G)$
for the weak-* topology.
\end{proof}

\begin{cor} If $G$ is finite, then $\mathcal{L}_{\alpha}(G) = \mathcal{M}^{\circ}_{\alpha}(G) = \mathbf{M}_{\alpha}(G)$.
\end{cor}

Finally, we identify $\mathcal{M}^{\circ}_{\alpha}(G)$ with $\mathcal{M}_{\alpha,\tau}(G)$
when $G$ is finite.  Recall that $\mathcal{A}(G) = C(G)$ is the dual of the group
algebra $\C G = \bigoplus_{u \in \widehat{G}} \gl(u)$,
and is identified to it by the Haar measure and the biduality between
elements on $\C G$ and measures on $G$.
This classically identifies $\mathcal{A}_{u^*}(G)$
with $\gl(u)$ (see e.g. \cite{HR2} \S 27, notably 27.49 (b)).
As a consequence this also identifies $\mathcal{A}_{u^*}(G)^+$ with $\gl((u^* \otimes \alpha)\circ \tau)$
hence $\gl'_{\alpha}(\tilde{u}) = \gl_{\alpha}(u)$. Finally, $u^* \simeq (\alpha \otimes u^*)^* \circ \tau
\Leftrightarrow u \circ \tau \simeq \alpha \otimes u^*$ and $\osp_{\alpha}(u) = \osp_{\alpha}'(u)$,
which shows that $\mathcal{M}_{\alpha}^{\circ}(G) = \mathcal{M}_{\alpha,\tau}(G)$.
This provides the second proof of the theorem :
\begin{cor} \label{corharm} If $G$ is finite, then $\mathcal{L}_{\alpha,\tau}(G) = \mathcal{M}_{\alpha,\tau}(G)$.
\end{cor}

\end{document}